\documentclass[11pt]{amsart}
\usepackage{amssymb,amsfonts,geometry}
\usepackage{hyperref}
\usepackage{mdwlist}
\usepackage{amssymb,textcomp}
\usepackage{enumitem}
\usepackage[utf8]{inputenc}
\usepackage[T1]{fontenc}
\usepackage{tikz, tikz-cd}

\usetikzlibrary{snakes}
\usetikzlibrary{math}
\usepackage{tikz}
\usepackage{hyperref}
 \hypersetup{
     colorlinks=true,
     linktocpage=true,
     linkcolor=red,
     filecolor=blue,
     citecolor = blue,
     urlcolor=cyan,
     }
\usepackage[capitalize]{cleveref}     

\usetikzlibrary{automata,matrix,positioning}
\usepackage{pst-node}
\usepackage{tikz-cd} 

\usepackage{comment}
\newtheorem{lemma}{Lemma}[section]
\newtheorem{theorem}[lemma]{Theorem}
\newtheorem*{theorem*}{Theorem}

\newtheorem{corollary}[lemma]{Corollary}

\newtheorem*{proposition*}{Proposition}

\newtheorem*{problem*}{Problem}

\theoremstyle{definition}
\newtheorem*{claim*}{Claim}

\newtheorem{definition}{Definition}

\newtheorem{example}{Example}

\newtheorem{remark}{Remark}
\newtheorem{remarks}{Remarks}

\newcommand{\N}{{\mathbb N}}

\renewcommand{\S}{\mathbb{S}}

\newcommand{\Z}{{\mathbb Z}}

\newcommand{\CA}{{\mathcal A}}

\newcommand{\CD}{{\mathcal D}}

\newcommand{\CF}{{\mathcal F}}
\newcommand{\CG}{{\mathcal G}}

\newcommand{\CK}{{\mathcal K}}
\newcommand{\CL}{{\mathcal L}}
\newcommand{\CM}{{\mathcal M}}

\newcommand{\CS}{{\mathcal S}}

\title{A note on the perturbations of subshifts}
\author{Haritha Cheriyath }
\address{Centro de Modelamiento Matemático (CNRS IRL2807)\\Universidad de Chile\\Santiago, Chile}
\email{hcheriyath@cmm.uchile.cl,harithacheriyath@gmail.com}

\thanks{The author was supported by ANID/Fondecyt/3250410 and Centro de Modelamiento Matemático (CMM) FB210005, BASAL funds for centers of excellence from ANID - Chile}

\subjclass[2020]{37B10 (Primary); 37B40 (Secondary)}

\keywords{Coded systems, sofic shifts, subshifts of finite type, entropy, conjugacy of subshifts}

\begin{document}

\begin{abstract}
  In this paper, we consider different classes of subshifts and study their perturbations obtained by forbidding sequences that contain a given word as a subword. 
  We show that the perturbations of sofic shifts are sofic. Though not true for general coded systems, we provide sufficient conditions under which perturbations of synchronized systems remain synchronized. We investigate which systems remain conjugate under perturbation by a word from a coded system.
We also explicitly calculate the drop in entropy under perturbation when the subshift is an $\CS$-gap shift, and obtain its exponential decay as the length of the word increases. In the case of subshifts of finite type, we explore the drop in entropy under multi-word perturbations.
\end{abstract}

\maketitle

\section{Introduction}
In the study of symbolic dynamics, one often considers perturbations of subshifts by forbidding specific words~\cite{CMRW24,Lind,Pavlov11,Ramsey24}. This simple question of deleting a single word from a subshift yields a diverse range of unexpected connections, for instance, in coding theory~\cite{Guibas}, non-transitive games~\cite{Combinatorial,Penney69} or ergodic theory~\cite{Subshift}. Let $X$ be a one-dimensional subshift on finite symbols, and let $w\in\CL(X)$ be a word appearing in the language of $X$. We define the perturbed system associated with $w$ as,
\[
X_w:=\{x\in X\mid x \text{ does not contain } w \text{ as a subword}\}.
\]

Clearly, $X_w\subseteq X$ is itself a subshift. When $X$ is minimal, all perturbations are empty. In the case of non-minimal subshifts, this basic perturbation prompts several natural questions: namely, what guarantees $X_w\ne\emptyset$, how to compute the entropy of $X_w$, how does the entropy depend on the length of $w$, or for two distinct words $w$ and $u$, when is $X_w$ conjugate to $X_u$. This paper investigates these problems, as well as further structural aspects of the perturbations, in the case where $X$ is a coded system. The results in this paper serve as preliminary steps toward addressing many future problems, and they generalize those in~\cite{CMRW24,Lind,Ramsey24}.

Perturbation by one or multiple words has been studied earlier when the ambient space is a \emph{subshift of finite type} (SFT in short), but has not been extended beyond this setting. It is easy to verify that if $X$ is an SFT, so is $X_w$. One can ask under what conditions does a structural property of $X$ persist in the perturbed system $X_w$. We address this problem when $X$ is a \emph{coded system}. 
Coded systems were introduced by Blanchard and Hansel~\cite{BH86} and form an important class of subshifts due to their properties and diverse applications~\cite{EKO19,Fiebig92,Krieger00,Pavlov20}. They are the subshifts represented by countable-state irreducible labeled graphs. Well-known examples of coded systems include irreducible sofic shifts, $\CS$-gap shifts, $\beta$-shifts, and Dyck shifts. We provide examples to show that perturbations of coded systems may not be coded. However, we prove the following result when $X$ is synchronized. 
\begin{theorem}\label{thm:synchro}
     Let $X$ be a synchronized system with a synchronizing word $m$. If $w\in\CL(X)$ is such that $X_w$ is irreducible, $m$ is not a subword of $w$, and $m\in\CL(X_w)$, then $X_w$ is a synchronized system.
\end{theorem}


Since the sets of admissible words of length $n$ satisfy $\CL_n(X_w)\subseteq\CL_n(X)$ for all $n\in\N$, the corresponding topological entropies satisfy $h(X_w)\le h(X)$.
From \cite{Lind}, for an SFT $X$ with entropy $\ln(\lambda)$, the difference $h(X)-h(X_w)$ decays exponentially in length $n=|w|$. 
A similar exponential bound for multidimensional SFT is obtained in~\cite{Pavlov11}.
However, this behavior does not necessarily hold for general subshifts.
To illustrate, consider minimal subshifts $X$ and $Y$ on distinct alphabets with $h(X)<h(Y)$. Then, $Z:=X\sqcup Y$ is non-minimal and $h(Z)=h(Y)$. For any word $w\in\CL(Y)$, the perturbed subshift satisfies $h(Z_w)=h(X)$, regardless of $|w|$. Although $Z$ is not irreducible, it demonstrates that convergence of entropies can fail in general settings.
This motivates the broader question of characterizing those subshifts $X$ for which $h(X_w)\to h(X)$ as $|w|\to\infty$, and further, identifying when this convergence occurs at an exponential rate. 
We prove the following for the perturbations of sofic shifts. 
\begin{theorem}\label{thm:sofic_entropy}
    Let $X$ be an irreducible sofic shift. There exist $C(X),N(X)>0$ such that for any $w\in\CL(X)$ with $|w|=n>N(X)$, 
    \[
    h(X)-h(X_w)\le \frac{C(X)}{n}.
    \]  In particular, $h(X_w)$ converges to $h(X)$ as $|w|\to\infty$.
\end{theorem}
Irreducible sofic shifts form a special class of coded systems that are also natural generalization of SFTs. It was introduced by Weiss~\cite{Weiss73} and is the smallest collection of subshifts that includes SFTs and is closed under taking factors. In automata theory, they are the subshifts represented by finite-state automata.
In~\cref{thm:sofic}, we prove that perturbations of sofic shifts are sofic.
Since irreducible sofic shifts can be well approximated by irreducible SFTs, studying perturbations of sofic shifts essentially reduces to studying those of SFTs. 
We also study this convergence problem when $X$ is an $\CS$-gap shift. 

\begin{theorem}\label{thm:Sgap_exp}
Let $\CS\subseteq \N_0$ and $X=X_{\CS}$ be the $\CS$-gap shift with entropy $\ln(\lambda_{\CS})$. Then for $w\in \CL_n(X)$ where $w\ne 0u0$ for some word $u$ containing 1, we have
     $h(X)-h(X_w)=O(\lambda_{\CS}^{-n})$ for $n$ large enough.
\end{theorem}
The class of $\CS$-gap shifts is an important subclass of coded systems, with applications in computer science and coding theory~\cite{DJ12,LM_book}.
In~\cref{thm:S_gap}, we explicitly compute the drop in entropy of systems perturbed from $\CS$-gap shifts which is used to prove the exponential decay stated above. 
Currently, the case where $w$ begins and ends with $0$ remains unresolved, except when $\CS=\{kd\mid k\in\N_0\}$ for some $d\ge 1$ (this case is discussed separately in Section~\ref{subsec:d-gap}). 

Another aspect of interest concerns the dynamical equivalence of different perturbations. Specifically, given two words $u,w\in\CL_n(X)$, when are the perturbed subshifts $X_u$ and $X_w$ conjugate? A necessary condition would be that they both have the same entropy. 
The correlation polynomial of a word determines the overlap of the word with itself. 
It has been proven to be a fundamental object when it comes to computing the entropy of the system perturbed from an SFT (\cref{eq:SFT_one}).
When the ambient system is full shift, the entropies of the perturbed systems are the same if and only if their corresponding correlation polynomials are the same. Hence, a necessary condition for two perturbed systems to be conjugate is that they share the same correlation polynomial. A word $w$ is called \emph{prime} if its correlation polynomial is given by $z^{|w|-1}$. In~\cite[Theorem 4.3]{CMRW24}, it was proved that when $X$ is a full shift, perturbations by prime words are conjugate.
When $X$ is coded, along with the correlations, it also depends on the graph that represents $X$. 
The following result generalizes~\cite[Propositions 4.11 \& 4.13]{CMRW24} where $X$ is a full shift or a golden mean shift, respectively.  
Let $\CG$ be a labeled graph representing $X$. We define $S_{\CG}(w)$ (respectively, $R_{\CG}(w)$) as the set of initial (respectively, terminal) vertices of all finite walks labeled $w$ in $\CG$ (defined explicitly in~\cref{sec:coded}). We refer to~\cref{def:corr} for the notations on the correlation. 

\begin{theorem}\label{thm:conjugacy}
     Let $X$ be a coded system, and $u,w\in\CL(X)$ be two words of the same length, say $n$, such that $(u,u)=(w,w)$ and $(u,w)=(w,u)=(u,u)\setminus\{n-1\}$. If there exists a presentation $\CG$ of $X$ such that $S_{\CG}(u)=S_{\CG}(w)$ and $R_{\CG}(u)=R_{\CG}(w)$, then $X_u$ and $X_w$ are conjugates. In particular, $h(X_u)=h(X_w)$.
\end{theorem}

Instead of perturbing a subshift with a single word, one can also ask the above questions on perturbation by multiple words. Let $\CK$ be a finite collection of allowed words in $X$ and let $X_\CK$ denote the collection of sequences from $X$ that do not contain words from $\CK$ as subwords. This multi-word perturbation of SFT was explored in~\cite{Ramsey24}. By studying the complexity function ($n\mapsto \#\CL_n(X_{\CK})$), the author obtained a bound on the entropy that is linear in the length of the shortest word. 
We use this bound in order to prove~\cref{thm:sofic_entropy}.
However, an exponential bound as in~\cite{Lind} is not yet obtained. In~\cref{thm:SFT_Multi}, we calculate the entropy of the system perturbed by multiple words. This expression demonstrates exponential entropy decay in the perturbed system (though we do not prove it explicitly). \emph{As a special case, when $X$ is a full shift on $N$ symbols, in~\cref{thm:Exp_decay}, we calculate the convergence of $\dfrac{h(X_{\CK_n})-h(X)}{N^{-n}}$ as $n\to\infty$ where $\CK_n$ is a finite collection of words of length $n$, which improves the result in~\cite{Ramsey24}.} We also give an application of this convergence in computing the escape rate into a shrinking disconnected hole. 

\section{Preliminaries and existing results}
\subsection{Symbolic dynamical systems}
 Let $\CA$ be a finite alphabet. A \emph{word} is a finite sequence (string) with symbols from $\CA$.
 If $w=w_1\dots w_n$ is a word, then $|w|=n$ denotes the length of $w$. 
 For two words $u,w$, we say that $u$ is a \emph{subword} of $w$, denoted as $u\prec w$, if there exist two words $t,v$ (possibly empty) such that $w=tuv$. 
 
 Let $\CA^{\Z}$ denote the collection of all bi-infinite sequences with symbols from $\CA$.  It is a compact metric space with respect to the pro-discrete topology. For $x,y\in\CA^{\Z},$ we define the metric, $d(x,y)=2^{-k}$ where $k=\min\{|i|: \ x_i\ne y_i\}$.
 Define the \emph{left shift map} $S:\CA^\Z\to \CA^\Z$ as $(S(x))_i=x_{i+1}$ where $x=(x_i)_{i\in\Z}$. A \emph{subshift} $X\subseteq \CA^\Z$ is a closed $S$-invariant subset, with the induced topology. When $X=\CA^\Z$ we call it a \emph{full shift}. 
 
 For a given subshift $X$, let $x=(x_i)_{i\in\Z}\in X$. For $i\le j$, we denote $x_{[i,j]}=x_ix_{i+1}\dots x_{j-1}x_j$. 
 A word $w$ with symbols from $\CA$ is said to be \emph{allowed} in $X$ if $w$ is a subword of a sequence in $X$, that is, there exists a sequence $x=(x_i)_{i\in\Z}\in X$ and $i\le j$ such that $w=x_{[i,j]}$. For $n\ge 1$, let $\mathcal{L}_n(X)$ denote the collection of all allowed words of length $n$ in $X$, and let $\CL_0(X)$ consist of the empty word. We define $\mathcal{L}(X)=\bigcup_{n\ge 0}\mathcal{L}_n(X)$ as the \emph{language} of $X$. Recall that the \emph{topological entropy} of a subshift $X$ is given by,
  \[
  h(X)=\lim_{n\to\infty}\frac{1}{n}\ln\#(\mathcal{L}_n(X)),
  \]
where $\#(.)$ denotes the cardinality of a set. The subshift $X$ is said to be \emph{minimal} if for a given $x\in X$, the orbit $\{S^nx\mid n\in\Z\}$ is dense in $X$. The subshift $X$ is said to be \emph{irreducible}, if for any $(u,w)\in\CL(X)^2$, there exists $v\in\CL(X)$ such that $uvw\in\CL(X)$. If $X$ is not irreducible, it is called \emph{reducible}.

Let $X,Y$ be two subshifts. We say that $X$ is \emph{conjugate} to $Y$ if there exists a homeomorphism $\phi:X\to Y$ such that $\phi\circ S=S\circ\phi$. In this case, the map $\phi$ is called a \emph{conjugacy map} from $X$ to $Y$. For $X,Y$ a \emph{sliding block code} $\psi:X\to Y$ is defined by some local rule as $y_i=\Psi(x_{[i-n,i+m]})$, for $i\in\Z$ where $\Psi:\CL_{n+m+1}(X)\to \CL_1(Y)$ is a block code for some fixed $m,n\ge 0$. A conjugacy between two subshifts is always a bijective sliding block code~\cite{LM_book}. When the sliding block code is surjective, $Y$ is a \emph{factor} of $X$.

Let $\CF$ be a finite collection of words with symbols from $\CA$. Define a subshift,
\[
X_{\CF}:=\{x\in\CA^{\Z}\mid x \text{ does not contain words from $\CF$ as subwords}\}.
\] 
A subshift $X$ is said to be a \emph{subshift of finite type (SFT)} if there exists a finite set $\CF$ of words of finite length with symbols from $\CA$, such that $X=X_{\CF}$. For example, when $\CA=\{0,1\}$ and $\CF=\{11\}$, the SFT $X_{\CF}$ that we obtain is the golden mean shift.

Let $G=(V,E)$ be a directed graph. 
We define an \emph{edge shift} on $G$ as,
\[
X_G=\{e\in E^{\Z}\mid r(e_i)=s(e_{i+1})\text{ for every } i\in\Z\}, 
\] where $s:E\to V$ and $r:E\to V$ are the source and the range map of edges. Clearly, $X_G$ is a subshift of finite type with the forbidden collection $\CF=\{ee'\mid e,e'\in E \text{ and } r(e)\ne s(e')\}$. Using the higher block code representation, it is not difficult to see that any given subshift of finite type is conjugate to an edge shift~\cite{LM_book}.  
Let $A$ be an irreducible matrix\footnote{That is, for every $i,j$ there exists $n$ such that $A^n_{i,j}>0$.} of non-negative integer entries and let $G_A$ be a directed graph such that the associated adjacency matrix is $A$. We denote $X_A$ to be the edge shift associated with $G_A$. It is well-known that $h(X_A)=\ln(\lambda_A)$ where $\lambda_A$ is the Perron root of $A$. 

A graph $G=(V,E)$ is \emph{irreducible} if for any pair of vertices $(u,v)\in V^2$, there is a walk from $u$ to $v$. 
The adjacency matrix of an irreducible graph is irreducible. 

\subsection{Perturbations of subshifts of finite type}

Let $G$ be a finite irreducible directed graph with adjacency matrix $A$, $X=X_A$ and $w\in\CL(X)$. We denote the perturbation of $X$ by $w$, denoted as $X_w$, to be the collection of all sequences in $X$ that do not contain $w$ as a subword. Clearly, $X_w\subseteq X$ is a subshift of finite type. From~\cite{Lind} we have the following expression for the entropy of $X_w$,
\begin{equation}\label{eq:SFT_one}
    h(X_w)=\ln(\lambda_w),
\end{equation} where $\lambda_w$ is the largest real zero of 
\[
\det(zI-A)c_w(z)+Adj_{i,j}(zI-A).
\]
Here, $I$ is the identity matrix of the respective size, $i=s(w)$ and $j=r(w)$ are the source and range of the walk $w$ in $G$, $Adj_{i,j}(M)$ is the $i,j$-th entry of the adjoint of the matrix $M$, and $c_w(z)$ is the correlation polynomial of $w$ as defined below. 
\begin{definition}\label{def:corr}
	Let $u$ and $w$ be two words with symbols from $\CA,$ and $n=|u|$. The \textit{correlation polynomial} of $u$ and $w$ is defined as
	\[
	(u,w)_z = \sum_{\ell=0}^{n-1} b_\ell z^{\ell}, 
	\]
	where $b_{\ell}=1$, if and only if the overlapping parts of $u$ and $w$ are identical when the left-most symbol of $w$ is placed under the $(n-\ell)^{th}$ symbol of $u$ (from the left), and $b_{\ell}=0$ otherwise. When $u=w$, we denote $(w,w)_z=c_w(z)$. We also denote 
    \[
    (u,w)=\{\ell\mid 0\le\ell\le n-1 \text{ and } b_{\ell}=1\}.
    \]
    For a word $w$, if $(w,w)=\{|w|-1\}$, then we call $w$ a \emph{prime word}. For two distinct words $u,w$, when $(u,w)=(w,u)=\emptyset$, we say that they have \emph{empty cross-correlations}.
\end{definition}


Due to the expression of topological entropy in terms of the correlation polynomial, it is also not difficult to see that the entropy $h(X_w)=O(\lambda^{-n})$ for $n$ large enough, where $n=|w|$ and $h(X)=\ln(\lambda)$~\cite[Theorem 3]{Lind}.  

\section{Perturbations of coded systems}\label{sec:coded}
\begin{definition}
   For a given finite alphabet $\CA$ and a countable irreducible directed graph $G=(V,E)$, we define the edge label as a map $\ell:E\to\CA$. The tuple $\CG=(G,\ell)$ is a labeled directed graph. We define a subshift $X_{\CG}$ to be the closure of
\[
\{x\in\CA^{\Z}\mid \exists e\in E^{\Z} \text{ such that } r(e_{i})=s(e_{i+1}) \text{ and } \ell(e_i)=x_i,\ \forall i\in\Z\}.
\] 
A subshift $X$ is said to be \emph{coded} if there exists a labeled countable irreducible directed graph $\CG$ such that $X=X_{\CG}$. In this case, we say that $\CG$ is a \emph{presentation} of $X_{\CG}$. 
\end{definition}
 Clearly, coded systems are irreducible. When $G$ is a finite directed graph (not necessarily irreducible), we call the subshift a \emph{sofic shift}. For a finite directed graph where the labeling is identity, $\ell:E\to E$, the resulting sofic shift is a subshift of finite type. 
An irreducible sofic shift has an irreducible labeled graph representation and hence is a coded system. Hence, all the irreducible SFTs are also coded systems. 

\begin{example}\label{ex:s-gap}
    Another notable example of coded systems is \emph{$\CS$-gap shifts}. Let $\CS=\{s_i\in\N_0\mid s_i<s_{i+1},i\ge 0\}$ be a nonempty subset of non-negative integers. If $\CS$ is finite, define the $\CS$-gap shift, denoted as $X_{\CS}$, to be the set of all
binary sequences for which 1's occur infinitely often in both directions and the number of 0's between successive occurrences of a 1 is an integer in $\CS$. When $\CS$ is infinite, together with the sequences defined above, we also allow sequences to begin or end with an infinite string of zeros. For example, when $\CS=\N_0$, $X_{\CS}$ is the full shift on 2 symbols, when $\CS=\N_0\setminus\{0\}$, $X_{\CS}$ is the golden mean shift and when $\CS=\{2k\mid k\in\N_0\},$ $X_{\CS}$ is the even shift. In general, it is easy to see that $X_\CS$ is a subshift. In fact, when $\CS$ is finite, then $X_\CS$ is an SFT and when $\Delta_S$ is eventually periodic, then $X_\CS$ is sofic where $\Delta_\CS=\{s_{i+1}-s_{i}\}_{i\ge 0}$~\cite{DJ12}. It is also a coded system and has a presentation given as $G=(V,E)$ where $V=\{v_0,v_1,\dots\}$ and there is an edge from $v_i$ to $v_{i+1}$ labeled 0 for all $i\ge 0$ and there is an edge from $v_i$ to $v_0$ labeled 1 whenever $i\in\CS$. We refer to~\cref{fig:2} for an example of graph representation for the even shift. 
\begin{figure}[h]
		\centering
		$\displaystyle
            \begin {tikzpicture}[-latex ,auto ,node distance =2cm and 3cm ,semithick , state/.style ={draw, circle}] 
		\node[state,scale=0.9] (A) {$v_0$};
		\node[state,scale=0.9] (B) [right of= A] {$v_1$};
        \node[state,scale=0.9] (C) [right of= B] {$v_2$};
        \node[state,scale=0.9] (D) [right of= C] {$v_3$};
        \node[state,scale=0.9] (E) [right of= D] {$v_4$};
        \node[state,scale=0.9] (F) [right of= E] {$v_5$};
		\path (A) edge [ left =25] node[above] {$0$} (B);
        \path (B) edge [ left =25] node[above] {$0$} (C);
        \path (C) edge [ left =25] node[above] {$0$} (D);
        \path (D) edge [ left =25] node[above] {$0$} (E);
        \path (E) edge [ left =25] node[above] {$0$} (F);
	\path (C) edge [bend right = -30] node[below] {$1$}(A);
    \path (E) edge [bend right = -30] node[below] {$1$}(A);
		\path (A) edge[in=145,out=-145,loop,distance=2cm]node[left]{$1$} (A);
        \node[right=.1cm of F] {$\dots$};
	\end{tikzpicture} 
	$
	\caption{Labeled graph representation for even shift}
	\label{fig:2}
\end{figure}
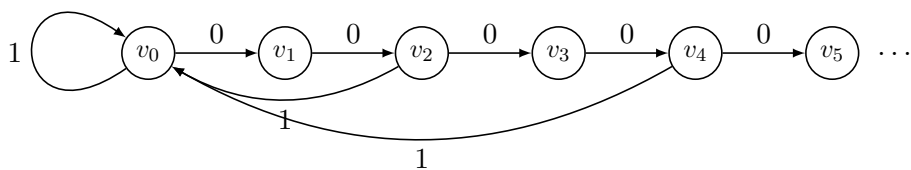
\end{example}

In general, for a coded system $X$ and an allowed word $w \in \mathcal{L}(X)$, consider the perturbed system $X_w$. In the simplest case, where $X$ is the full shift on $\CA = \{0,1\}$ and $w = 10$, we have $0, 1 \in \CL(X_w)$ since $\overline{0}, \overline{1} \in X_w$, where $\overline{a} = \dots aa.aa \dots$ denotes the bi-infinite repetition of the symbol $a$. However, there is no word $u \in \CL(X_w)$ such that $1u0 \in \CL(X_w)$. Hence, $X_w$ is reducible and therefore not coded.
Thus, $X_w$ may not always be a coded system. However, one may ask whether it remains coded under the additional assumption that $X_w$ is irreducible. We provide the following counterexample.

 Let $Y$ be an infinite minimal subshift on $\CA=\{0,1\}.$ For instance, we can take $Y$ to be the Fibonacci shift. Fix $x=(x_i)_{i\in\Z}\in Y$. We construct an irreducible labeled and countable graph $\CG$ as follows. Let $V=\Z$ and there is an edge from $i$ to $i+1$ labeled $x_i$ and there is an edge from $i+1$ to $i$ labeled 2 (\cref{fig:3}). Here, $X=X_{\CG}$ is a coded system. Now take $w=2$. Note that $X_w=Y,$ which is irreducible and minimal. This is not a coded system as $Y$ has no periodic points, whereas coded systems have a dense set of periodic points~\cite{BH86}. Hence, the irreducibility of $X_w$ is not a good enough assumption.  
 \begin{figure}[h]
		\centering
		$\displaystyle
            \begin {tikzpicture}[-latex ,auto ,node distance =2cm and 3cm ,semithick , state/.style ={draw, circle}] 
		\node[state,scale=0.9] (A) {$-2$};
		\node[state,scale=0.9] (B) [right of= A] {$-1$};
        \node[state,scale=0.9] (C) [right of= B] {$0$};
        \node[state,scale=0.9] (D) [right of= C] {$1$};
        \node[state,scale=0.9] (E) [right of= D] {$2$};
        \node[state,scale=0.9] (F) [right of= E] {$3$};
		\path (A) edge [ left =25] node[above] {$x_{-2}$} (B);
        \path (B) edge [ left =25] node[above] {$x_{-1}$} (C);
        \path (C) edge [ left =25] node[above] {$x_0$} (D);
        \path (D) edge [ left =25] node[above] {$x_1$} (E);
        \path (E) edge [ left =25] node[above] {$x_2$} (F);
	\path (B) edge [bend right = -30] node[below] {$2$}(A);
    \path (C) edge [bend right = -30] node[below] {$2$}(B);
    \path (D) edge [bend right = -30] node[below] {$2$}(C);
    \path (E) edge [bend right = -30] node[below] {$2$}(D);
    \path (F) edge [bend right = -30] node[below] {$2$}(E);
        \node[left=.1cm of A] {$\dots$};
        \node[right=.1cm of F] {$\dots$};
	\end{tikzpicture} 
	$
	\caption{Labeled graph representation for a coded system with a non-coded perturbation}
	\label{fig:3}
\end{figure}
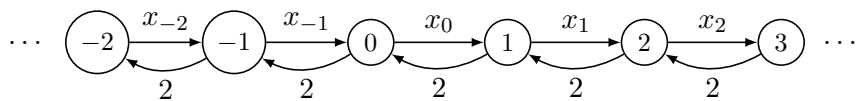

For an irreducible subshift $X$, an allowed word $m$ is said to be \emph{synchronizing} if whenever $um,mv\in\CL(X)$, we have $umv\in\CL(X)$. We call $X$ a \emph{synchronized system} if it has a synchronizing word in its language. All synchronized systems are coded~\cite{BH86}.  Now we prove~\cref{thm:synchro}.

\begin{proof}[Proof of~\cref{thm:synchro}] 
We prove that $m$ is a synchronizing word for $X_w$. Let $u,v\in\CL(X_w)$ be such that $um,mv\in\CL(X_w)$. Choose $u_n,v_n$ to be such that $u_num,mvv_n\in\CL(X_w)$ where $|u_n|=|v_n|=n$. In particular, we have $u_num,mvv_n\in\CL(X)$. Since $m$ is a synchronizing word for $X$, we have $u_numvv_n\in\CL(X)$. 
That is, there exists $x^n\in X$ such that $x^n_{[-n-|u|,n+|m|+|v|-1]}=u_numvv_n$. Here $w\not\prec u_numvv_n$ since $m\not\prec w$. Taking a converging subsequence of $(x^n)_{n\ge 1}$ and its limit $x,$ we find that $x_{[-|u|,|m|+|v|]}=umv$ and that $x$ does not contain $w$ as a subword. Hence, $x\in X_w$ and $umv\in\CL(X_w)$.
\end{proof}
\begin{example}\label{ex:synch}
Let $X_{\CS}$ be an $\CS$-gap shift. Here, $1$ is a synchronizing word. The only choices for $w\in\CL(X_{\CS})$ that do not contain 1 is when $w=\overline{0}^k$ (here $\overline{0}^k$ denotes 0 repeated $k$ times) for some $k\ge 1$. In this case, it is easy to see that $X_w=X_{\CS_0}$ where $S_0=\{s\in\CS\mid s< n\}.$ Thus, $X_w$ is an irreducible SFT and is therefore synchronized. This is a very special case. If $k\in\CS$, then $m=1\overline{0}^k1\in\CL(X_{\CS})$ is also a synchronizing word. For any word $w\in\CL(X_{\CS})$ that does not contain $m$ as a subword, we see that $X_w$ is a synchronizing system provided it is irreducible. 
For instance, we can choose $k=min_{i}\{i\in\CS\}$. 
This gives us a lot more non-trivial examples.

The following example shows that $X_w$ being irreducible is a necessary condition in~\cref{thm:synchro}. That is, choosing $w\in\CL(X)$ so that $m\not\prec w$ and $m\in\CL(X_w)$ do not guarantee that $X_w$ is irreducible. Let $X$ be the full shift on two symbols. Then we can choose $m=11$. Let $w=10$. Here, $m\not\prec w$ and $m\in \CL(X_w)$ since $\overline{1}\in X_w$. But $X_w$ is not irreducible. 
\end{example}

Let $X$ be a coded system and let $\CG=(V,E,\ell)$ be a presentation of $X$. We define $S_{\CG}(w)=\{v\in V\mid s(w)=v\}$ and $R_{\CG}(w)=\{v\in V\mid r(w)=v\}$ where $s(w),r(w)$ are the source and range of all possible walks labeled $w$ in $\CG$, respectively. We will now prove~\cref{thm:conjugacy}.

\begin{proof}[Proof of~\cref{thm:conjugacy}]
Let $\CA$ be the alphabet and $|u|=|w|=n$.
    The proof is similar to that of~\cite[Proposition 4.1]{CMRW24}. We define a map $\phi:X_u\to X_w$ that replaces all the occurrences of $w$ in $X_u$ by $u$. 
    The map $\phi$ is a sliding block code induced by $\Phi:\CA^{2n-1}\to\CA$ defined as,
    \[
    \Phi(x_{[-n+1,n-1]})=\begin{cases}
        u_{i}\ \text{ if } x_{[1-i,n-i]}=w \text{ for some } 1\le i\le n\\
 w_{i}\ \text{ if }  x_{[1-i,n-i]}=u \text{ for some } 1\le i\le n\\
        x_0\ \text{ otherwise}.
    \end{cases}
    \] One can easily verify that $\Phi$ is well-defined because the conditions on the correlations of $u,w$. Consider the sliding block code, $\phi':\CA^{\Z}\to\CA^{\Z}$. 
    For any sequence $x\in\CA^{\Z}$, the map $\phi'$ replaces the appearances of $u$ in $x$ by $w$ and $w$ in $x$ by $u$. Clearly, $\phi'$  is an involution. 

    It is enough to show that the restriction map $\phi:X_u\to\CA^{\Z}$ is mapped onto $X_w$. Let $\CG$ be the representation of $X$ such that $S_{\CG}(u)=S_{\CG}(w)$ and $R_{\CG}(u)=R_{\CG}(w)$. For $x\in X_u$, $x$ does not contain $u$ as a subword, but it may contain $w$ as its subword. Moreover,
    either $x$ is represented by a bi-infinite walk in $\CG$ or it is a limit point of the set of all bi-infinite walks in $\CG$. If $x$ is represented by a walk in $\CG$, since $S_{\CG}(u)=S_{\CG}(w)$ and $R_{\CG}(u)=R_{\CG}(w)$, replacing $w$ by $u$ is still an admissible walk in $\CG$. Hence $\phi(x)\in X_w$. Let $(x_i)_{i\ge 0}$ be a sequence of bi-infinite walks in $\CG$, each of them does not contain $u$ as a subword, such that $\lim_{i\to\infty}x_i=x$. For any given $N$, there exists $k$ such that for all $i>k$, $x_{[-N,N]}=x_{i_{[-N,N]}}$. Let $y_i=\phi(x_i)\in X_w$. Clearly, $\phi(x)\in X_w$ as it is a limit point of $(y_i)_{i\ge 0}$. Using a similar argument, we can show that $\phi(X_w)\subseteq X_u$ and the result follows as $\phi$ is an involution.
\end{proof}
\begin{remark}
    For a subshift $X$ and two words $u,w\in\CL(X),$ the concept of replacing $u$ by $w$ in a sequence is related to their corresponding extender sets~\cite{GP19}. In order to ensure that the sequence that we obtain by replacing $u$ by $w$ is still admissible, it is enough to assume that they have the same extender set. 
\end{remark}
\begin{example}\label{ex:conjugacy}
\begin{enumerate}
    \item When $X$ is the full shift over the alphabet $\CA = \{1, \dots, N\} $, it can be represented by a labeled graph $\CG = (V, E, \ell)$, where $V = \{v_0\}$ consists of a single vertex, and there are $N$ edges from $v_0$ to itself, each labeled with a distinct symbol from $\CA$ (When $N=2,$ we refer to the left most graph of~\cref{fig:1}). Hence, any two words $u,w\in\CA^n$ satisfying the required conditions on their correlations as in~\cref{thm:conjugacy} give two perturbed systems $X_u$ and $X_w$ that are conjugate.

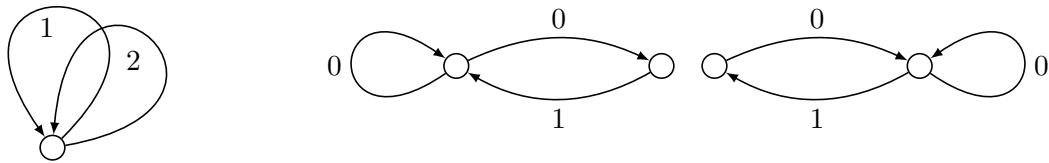
\begin{figure}[h]
		\centering
		$\displaystyle
        \begin{matrix}
            \begin {tikzpicture}[-latex ,auto ,node distance =3cm and 4cm ,semithick , state/.style ={draw, circle}] 
		\node[state,scale=0.9] (A) {};
		\path (A) edge [in=125,out=45,loop,distance=3cm] node[auto]{$1$} (A); 
        	\path (A) edge [in=85,out=10,loop,distance=3cm] node[auto]{$2$} (A); 
	\end{tikzpicture}
            &
            \begin {tikzpicture}[-latex ,auto ,node distance =3cm and 4cm ,semithick , state/.style ={draw, circle}] 
		\node[state,scale=0.9] (A) {};
		\node[state,scale=0.9] (B) [right of= A] {};
		\path (A) edge [bend left =25] node[above] {$0$} (B);
		\path (B) edge [bend right = -30] node[below] {$1$}(A);
		\path (A) edge[in=145,out=-145,loop,distance=2cm]node[left]{$0$} (A);
	\end{tikzpicture} 
    &
       \begin {tikzpicture}[-latex ,auto ,node distance =3cm and 4cm ,semithick , state/.style ={draw, circle}] 
		\node[state,scale=0.9] (A) {};
		\node[state,scale=0.9] (B) [right of= A] {};
		\path (A) edge [bend left =25] node[above] {$0$} (B);
		\path (B) edge [bend right = -30] node[below] {$1$}(A);
		\path (B) edge[in=35,out=-35,loop,distance=2cm]node[right]{$0$} (B);
	\end{tikzpicture}
        \end{matrix}
	$
	\caption{Representations for full shift on 2 symbols (left) and golden mean shift (center and right)}
	\label{fig:1}
\end{figure}

   \item When $X$ is the golden mean shift and $u,w\in\CL_n(X)$ both start with 0 and end with 1, one can choose $\CG$ as the middle graph in~\cref{fig:1} and we get that any two such words $u,w$  satisfying the required conditions for their correlations as in~\cref{thm:conjugacy} give two perturbed systems $X_u$ and $X_w$ that are conjugate. Similarly, when $u$ and $w$ both start with 1 and end with 0, we can consider $\CG$ as in the right-most graph in~\cref{fig:1}.   
\end{enumerate}
   Both of the above examples were studied in~\cite{CMRW24}, and they called this conjugacy a \emph{swap conjugacy}. When two prime words have a nonempty cross-correlation, a chain of swap conjugacies between them can be constructed by introducing intermediate prime words with empty cross-correlation. However, extending this approach to a general coded system appears to be quite difficult.
   
\begin{enumerate}
    \item[(3)] We can apply the above result to an $\CS$-gap shift. Let $u$ be a word that begins with $m>0$ zeros and ends with 1. Then for the representation $\CG$ as described in~\cref{ex:s-gap}, $S_{\CG}(u)=\{v_{t-m}\mid t\in\CS, \ t\ge m\}$ and $R_{\CG}(u)=\{v_0\}.$ Similarly, if $u'$ is a word that begins with 1 and ends with $m'>0$ zeros, then for the representation $\CG$ as described in~\cref{ex:s-gap}, we have $S_{\CG}(u)=\{v_{t}\mid t\in\CS\}$ and $R_{\CG}(u)=\{v_m\}.$ Hence, by~\cref{thm:conjugacy}, for two words $u,w\in\CL_n(X_\CS)$ with $0\notin (u,u)=(w,w)$ and $(u,w)=(w,u)=(u,u)\setminus\{n-1\}$, we have a conjugacy between $X_u$ and $X_w$ if they start with the same number of zeros and end with the same number of zeros. We will also see later that these perturbations have the same entropy.   
\end{enumerate}
\end{example}

\subsection{Perturbations of sofic shifts}\label{sec:sofic}
In this section, we restrict ourselves to the case where $X$ is a sofic shift. Let $\CG=(G,\ell)$ be a finite labeled directed graph that represents $X$. 

Let $\ell_\infty:X_{G}\to X_{\CG}$ be the sliding block code defined as $\ell_\infty(y)_i=\ell(y_i)$, where the block code is $\ell:\CL_1(X_G)\to\CL_1(X_\CG)$. This is clearly surjective and hence a factor map. In fact, a subshift is sofic if and only if it is a factor of a subshift of finite type~\cite{LM_book}. 

A labeled directed graph is called \emph{right-resolving} if distinct edges in $G$ starting from the same vertex are labeled distinctly. 
A sofic shift $X$ is said to have a right-resolving presentation if there exists a right-resolving labeled graph $\CG=(G,\ell)$ such that $X=X_{\CG}$. Any sofic shift $X$ has a right-resolving presentation $\CG=(G,\ell)$ where $h(X)=h(X_G)$. One can choose $\CG$ to be an irreducible right-resolving graph when $X$ is also irreducible.
We can naturally extend the map $\ell:\CL(X_G)\to\CL(X_\CG)$.
If we choose $\CG$ to be right-resolving, for a word $w\in\CL_n(X_{\CG}),$ there are at most $r$ pre-images in $\CL_n(X_G)$ by $\ell$ where $r$ is the number of vertices in $G$. 

\begin{theorem}\label{thm:sofic}
    Let $X$ be a sofic shift and $w\in\CL(X)$. Then $X_w$ is a sofic shift. 
\end{theorem}
\begin{proof}
    We assume $X_w\ne\emptyset$.
    Let $\CG=(G,\ell)$ be a right-resolving presentation of $X$ such that $h(X)=h(X_G)$. There exists a factor map $\ell_\infty:X_G\to X$.  We denote $Y=X_G$. Let $\ell^{-1}(w)=\{u_1,\dots,u_s\}$ and let $Y_{u_1,\dots,u_s}\subseteq Y$ be the subshift that consists of sequences that do not contain $u_1,\dots,u_s$ as subwords. 
    
    We consider the restriction map $\ell_\infty:Y_{u_1,\dots,u_s}\to X$. If $y\in Y$ does not contain $u_1,\dots,u_s$, then $\ell_\infty(y)$ does not contain $w$. Hence, $\ell_\infty(Y_{u_1,\dots,u_s})\subseteq X_w$. Also, for any sequence $x\in X_w$, its pre-image $\ell_\infty^{-1}(x)$ contains sequences that do not have $\ell^{-1}(w)$ as subwords. Hence, $\ell_\infty':Y_{u_1,\dots,u_s}\to X_w$ is a sliding block code that is surjective, and thus a factor map. Clearly, $Y_{u_1,\dots, u_s}$ is an SFT and therefore $X_w$ is sofic. 
\end{proof}

Although the above theorem states that perturbations of irreducible sofic shifts remain sofic, it does not specify the conditions under which they are irreducible. 
Note that one cannot expect to guarantee the irreducibility of the perturbed system simply by considering longer words. For example, in case of full shift $X$ on 2 symbols, for each length $k\ge 2$, there exists a word $w\in\CL_k(X)$ such that $X_w$ is reducible~\cite[Proposition 3.5]{CMRW24}.  

Theorem~\ref{thm:sofic} clearly demands us to look at the perturbation of SFTs by multiple words. For a subshift $X$, let $\CK\subseteq\CL(X)$ be a finite collection of allowed words. We define the multi-word perturbation of $X$ by $\CK$, denoted as $X_{\CK}$, to be the collection of all sequences in $X$ that do not contain words from $\CK$ as subwords. We discuss more about the multi-word perturbations of SFTs in Section~\ref{sec:multi-wordSFT}. However, we state the following result that allows us to study the perturbation of irreducible sofic shifts. 

\begin{theorem}\cite[Theorem 1]{Ramsey24}\label{thm:multi_SFTLinear}
Let $X$ be an irreducible subshift of finite type.
    There exist $C(X),N(X)>0$ such that for a collection of finite allowed words $\CK$ in $X$ with the length of the shortest word, say $n>N(X)$, we have
    \[
    h(X)-h(X_{\CK})\le 
    \frac{C}{n}.
    \]
\end{theorem}
Using the approximation of sofic shifts by SFTs,~\cref{thm:sofic_entropy} is immediate from~\cref{thm:multi_SFTLinear}. 

\begin{proof}[Proof of~\cref{thm:sofic_entropy}]
    Let $\CG=(G,\ell)$ be an irreducible, right-resolving presentation of $X$ such that $h(X)=h(Y)$ where $Y=X_G$.
    Let $\ell^{-1}(w)=\{u_{1},\dots,u_{s}\}$. If $r$ is the number of vertices of $G$, then $s\le r$. Also, $r\left(\#\CL_n(X_w)\right)\ge \#\CL_n(Y_{u_1,\dots,u_s})$ for all $n\ge 1$ implies that $h(Y_{u_{1},\dots,u_{s}})=h(X_{w}).$ Hence, the result follows from Theorem~\ref{thm:multi_SFTLinear}. 
\end{proof}

Again, to get an exponential decay, we need to obtain an exponential decay for the multi-word perturbation of SFT.

\section{Perturbations of $\CS$-gap shifts}\label{sec:S-gap}
Let $\CS=\{s_i\in\N_0\mid s_i<s_{i+1},i\ge 0\}$ be a nonempty subset of non-negative integers. Recall the definition of the $\CS$-gap shift $X_{\CS}$ from Example~\ref{ex:s-gap}. 
A sequence $x\in\{0,1\}^{\Z}$ belongs to $X_{\CS}$ if each subword of the form $1\overline{0}^n1$ that appears in $x$ satisfies $n\in\CS$.
We have $h(X_\CS)=\ln(\lambda_\CS)$~\cite{Spandl07} where $\lambda_\CS$ is the unique real root of
\[
\sum_{s\in\CS}\frac{1}{z^{s+1}}=1.
\]


For a given $\CS\subseteq\N_0$, let $X:=X_{\CS}$ and $w\in\CL(X)$. We define $X_w$ as the collection of all sequences that do not contain $w$ as a subword. Any $\CS$-gap shift is entropy-minimizing and hence $h(X_w)<h(X)$~\cite{GP19}.
As discussed in~\cref{ex:synch}, when $w=\overline{0}^n$ for $n\ge 1$, the perturbed system $X_w=X_{\CS_0}$ is an SFT where $S_0=\{s\in\CS\mid s< n\}.$ Hence, $h(X_w)=\ln(\lambda_w)$ where $\lambda_w$ is the unique real zero of
\[
\sum_{s\in\CS,\ s
< n}\frac{1}{z^{s+1}}=1.
\] Clearly, as $n\to\infty,$ the entropy $h(X_w)$ converges to $h(X)$.
We are interested in calculating $h(X_w)$ for general $w$. Before that, we present some notations. We define the characteristic function $\chi_{\CS}:\N_0\to\{0,1\}$ as $\chi_{\CS}(n)=1$ if and only if $n\in\CS$. 

For $w\in\CL(X)$ which is not of the type $\overline{0}^n$, let $\tilde{w}$ be the smallest word such that $w\prec \tilde{w}$ and $\tilde{w}=\overline{0}^ku\overline{0}^{k'}$ where $k,k'\in\CS\cup\{0\}$ and $u$ is a word that starts and ends with 1. For example, let $\CS=\{0,3,6,\dots\}$ and $w=00100010000$. Here, $\tilde{w}=00010001000000.$ Observe that, $X_w=X_{\tilde{w}}$.

In this section, we assume that $w\in\CL(X)$ does not start and end with 0 simultaneously (hence, the words of the type $\overline{0}^n$ are already ruled out). 
Let $\tilde{w}$ be as before. 
Assume $\tilde{w}$ begins with $m$ number of zeros or ends with $m$ number of zeros for $m\in\CS$ (note that $\tilde{w}$ may begin with 0, end with 0, or do neither; however, it cannot do both). When $w$ begins and ends with 1, we take $m=0$. We denote,
\begin{eqnarray*}
    T_{\CS}(z)&=&\sum_{n=0}^{\infty}\frac{\chi_{\CS}(n)}{z^{n+1}}\\
    \CS_m&=&\{n-m\mid n\in\CS, n\ge m\}.
\end{eqnarray*}

\begin{theorem}\label{thm:S_gap} 
Let $w\in\CL(X)$ be of the form $w=u\overline{0}^m$ or $w=\overline{0}^mu$ for some $u$ that starts and ends with 1 and $m\ge 0$. Then the generating function $F(z)=\sum_{n=0}^\infty f(n)z^{-n}$ where $f(n)=\#\CL_n(X_w)$ is given by,
    \[
    F(z)=\frac{z}{z-1}\frac{c_{\tilde{w}}(z)(T_{\CS^c}(z)+z)-T_{\CS^c_m}(z)}{(1-T_{\CS}(z))c_{\tilde{w}}(z)+T_{\CS_m}(z)},
    \] where $\CS^c=\N_0\setminus\CS$ and $\CS_m^c=\N_0\setminus\CS_m$.
\end{theorem}
\begin{proof} 
    Without loss of generality, assume that $w=\tilde{w}$. Let $\overline{w}$ denote the reverse of $w$, that is, $\overline{w}=w_n\dots w_1$ for $w=w_1\dots w_n$. Note that, there is a bijection from $\CL_n(X_w)$ to $\CL_n(X_{\overline{w}})$ by reversing. Hence $h(X_w)=h(X_{\overline{w}})$. It is also each to check that $c_w(z)=c_{\overline{w}}(z)$. This allows us to assume without loss of generality that $w$ starts with $m$ number of zeros for some $m\in\CS\cup \{0\}$ and ends with 1. 
    
    Let $f_\CS(n)$ be the number of words in $\CL_n(X_w)$ that end with $t$ number of zeroes for some $t\in\CS$, $f_{\CS_m}(n)$ be the number of words in $\CL_n(X_w)$ that end with $t$ number of zeroes for some $t\in\CS_m$, $f_1(n)$ be the number of words in $\CL_n(X_w)$ that end with 1, and $g_w(n)$ be the number of words in $\CL_n(X)$ in which the word $w$ occurs exactly once and it is at the end. Let $f_1(0)=1$, $f_\CS(0)=\chi_\CS(0)$, $f_{\CS_m}(0)=\chi_{\CS_m}(0)$ and $g_w(0)=0$. 

    Note that each word in $\CL_n(X_w)$ is either of the form $\overline{0}^{n}$ or $u\overline{0}^{n-k}$ where $u$ is a word in $\CL_k(X_w)$ that ends with 1. On the other hand, since $w$ ends with 1, for any $u\in\CL_k(X_w)$ that ends with 1, the word $u\overline{0}^{n-k}\in\CL_n(X_w)$. Hence,
    \[
    f(n)=\sum_{k=0}^nf_1(k).
    \]
    Also, each word that is counted in $f_\CS(n)$ is either of the form $\overline{0}^{n}$, if $n\in\CS$ or $u\overline{0}^{k}$ where $u$ is a word in $\CL_{n-k}(X_w)$ that ends with 1 and $k\in\CS$. On the other hand, since $w$ ends with 1, for each $k\in\CS$ and $u\in\CL_{n-k}(X_w)$ that ends with 1, the word $u\overline{0}^{k}\in\CL_n(X_w)$. Hence, 
    \[
    f_\CS(n)=\sum_{k=0}^nf_1(n-k)\chi_\CS(k).
    \] Similarly,
    \[
        f_{\CS_m}(n)=\sum_{k=0}^nf_1(n-k)\chi_{\CS_m}(k).
    \]
    Consider a word $u\in \CL_n(X_w)$. Note that $u0$ is a word in $\CL_{n+1}(X_w)$. Also, when $u$ ends with $t$ number of zeros for some $t\in\CS$, then $u1\in\CL_{n+1}(X_w)$ or $u1$ ends with $w$ and it is the only occurrence of $w$ in $u1$. The only word we do not count in $f_{\CS}(n)$ but contribute to a word in $f(n+1)$ or $g_w(n+1)$ is $\overline{0}^n1$ when $n\notin\CS$. Hence,
    \[
   f(n)+f_\CS(n)+ 1-\chi_\CS(n)=f(n+1)+g_w(n+1).
    \]
    In order to get the last recurrence relation, consider a word $u$ counted in $f_{\CS_m}(n)$ and attach $w$ at the end. 
    Firstly, note that this is an allowed word in $X$ since $t+m\in\CS$ for any $t\in\CS_m$. But this need not always be counted in $g_w(n+|w|)$ as other appearances of $w$ may occur in between. Let $uw=v_1\dots v_nv_{n+1}\dots v_{n+|w|}$ and let $k$ be the first occurrence of $w$ in $uw$, that is, $v_{k-|w|+1}\dots v_{k}=w$. Since $u\in\CL_n(X_w)$, $k>n$. Also, $w_{|w|-(k-n)+1}=w_1,\dots,w_{|w|}=w_{k-n}$. This implies that $k-n\in(w,w)$ and that this word is counted in $g_w(k)$. Conversely, for $k-n\in(w,w)$, any word counted in $g_w(k)$ is coming from a concatenation of a word counted in $f_{\CS_m}(n)$ and $w$, except for $\overline{0}^nw$ for $n\notin\CS_m$. Hence,
    \[
     f_{\CS_m}(n)+   1-\chi_{\CS_m}(n)=\sum_{t\in (w,w)}g_w(n+t).
    \]
    We obtain the result by multiplying each of these equations by $z^{-n}$ and summing for all $n\ge 0$ and then solving the linear equations. 
\end{proof}
\begin{corollary}\label{cor:S_gap}
    $h(X_w)=\ln(\lambda_w)$ where $\lambda_w$ is the largest real zero of $f_w(z)=(1-T_{\CS}(z))c_{\tilde{w}}(z)+T_{\CS_m}(z).$
\end{corollary}
\begin{proof}
    First, we prove that $f_w(z)$ has a root in the interval $(1,\lambda_{\CS})$ where $\lambda_{\CS}$ is the unique real root of $1-T_{\CS}(z).$ For all $z\in(1,\lambda_{\CS}),$ we have $T_{\CS_m}(z)\le\sum_{n=0}^{\infty}\frac{1}{z^{n+1}}=\frac{1}{z-1}.$ Hence $(z-1)f_w(z)\le (z-1)(1-T_{\CS}(z))c_{\tilde{w}}(z)+1.$ Denote this polynomial on the right hand side to be $g_w(z).$ Note that as $z\to 1^+,$ $g_w(z)$ converges to $-\infty.$ Hence, $f_w(z')$ is negative for some $z'\in(1,\lambda_{\CS})$. Also, at $z\ge\lambda_{\CS},$ we have $f_w(\lambda_{\CS})>0$. Hence,  $f_w(z)$ has a root in $(1,\lambda_{\CS}).$ Let $\lambda_w$ denote the largest real root of $f_w(z)$ in this interval. 

    If $z_0>1$ is a root of $f_w(z)$, then we show that $z_0$ is not a root for $c_{\tilde{w}}(z)(T_{\CS^c}(z)+z)-T_{\CS_m^c}(z).$ This will clearly imply that $\lambda_w$ is the largest real pole of $F(z)$ (notation as in~\cref{thm:S_gap}) and hence our result.  

Let us assume, on the contrary, that there exists $z_0>1$ such that
    \[
    (1-T_{\CS}(z_0))c_{\tilde{w}}(z_0)+T_{\CS_m}(z_0)=c_{\tilde{w}}(z_0)(T_{\CS^c}(z_0)+z_0)-T_{\CS^c_m}(z_0).
    \]
    Using the fact that $T_{\CS}(z)+T_{\CS^c}(z)=\frac{1}{z-1}=T_{\CS_m}(z)+T_{\CS_m^c}(z)$, the above equation can be reduced to 
    $(z_0^2-2z_0+2)c_{\tilde{w}}(z_0)=1$. As $z_0>1,$ we have $z_0^2-2z_0+2>1$  which implies that $c_{\tilde{w}}(z_0)<1$ which is not possible.
\end{proof}

\begin{remarks}
For $m\in\CS\cup\{0\}$, let
\[
C_{m,n}=\{w\in\CL_n(X_{\CS})\mid w \text{ starts with } m \text{ number of zeroes and ends with } 1 \}.
\]
We can easily conclude the following.
\begin{enumerate}
    \item For $u,w\in C_{m,n}$, when $c_{u}(z)=c_w(z)$, then $h(X_u)=h(X_w)$. This was also observed in~\cref{ex:conjugacy} (3).
    \item For $u,w\in C_{m,n}$, when $c_{u}(z)<c_w(z)$ for all $z\in(1,\lambda_{\CS})$, then $h(X_u)<h(X_w)$. In particular, the minimum is always achieved by the prime words.
\end{enumerate}
\end{remarks}

\noindent We now prove~\cref{thm:Sgap_exp}.
\begin{proof}[Proof of~\cref{thm:Sgap_exp}]
   We assume that $\CS$ is infinite, as otherwise $X_{\CS}$ is an SFT and it follows from~\cite[Theorem 3]{Lind}. 
   We first assume that $w\ne\overline{0}^k$ for some $k\ge 1$. 
   As before, let $\lambda_w$ be the largest root of $f_w(z)$ in the interval $(1,\lambda_{\CS}),$ where 
   \[
   f_w(z)=(1-T_{\CS}(z))c_{\tilde{w}}(z)+T_{\CS_m}(z).
   \] Cleary, $1-T_{\CS}(z)<0$ for all $z\in (1,\lambda_{\CS}).$
    For each $n$, let $h_n(z)=(1-T_{\CS}(z))z^{n-1}+\frac{1}{z-1}.$ Then $f_w(z)\le h_n(z)$ for all $z\in(1,\lambda_{\CS}).$ Let $z_n$ be the largest root of $h_n(z)$ in the interval $(1,\lambda_{\CS}).$ Note that such a $z_n$ exists and $z_n\le \lambda_w$. Also, since $h_n(z)\le h_1(z)$, we have $z_n\in[z_1,\lambda_{\CS}).$
    It is enough to show that $\lambda_{\CS}-z_n=O(\lambda_{\CS}^{-n})$ for $n$ large. 

    Since $\lambda_\CS$ is a simple and unique root of $1-T_\CS(z)$ that is larger than 1, we can write $1-T_\CS(z)=(z-\lambda_\CS)q(z)$ where $q(z)>0$ for all $z\in[z_1,\lambda_{\CS})$. Hence,
    \[
    0\le\lambda_\CS-z_n=\dfrac{1}{(z_n-1)q(z_n)z_n^{n-1}}\le Cz_n^{-n}
    \] where the constant $1/C$ is obtained taking the minimum of the function $(z-1)zq(z)$ on $[z_1,\lambda_\CS)$. 
Using the mean value theorem for $\ln(x)$ in $[z_n,\lambda_{\CS}]$, we have $-\dfrac{\ln(z_n)}{\ln(\lambda_{\CS})}<-1+dz_n^{-n}<-1+dz_1^{-n}$ for some $d>0$. Hence,
\[
0\le\lambda_{\CS}-\lambda_w\le\lambda_{\CS}-z_n\le C \lambda_{\CS}^{-n\frac{\ln(z_n)}{\ln(\lambda_{\CS})}}\le C\lambda_{\CS}^{-n}\lambda_{\CS}^{ndz_1^{-n}}\le C'\lambda_{\CS}^{-n}.
\] where $C'=C\max_{n\ge 1}\{\lambda_{\CS}^{ndz_1^{-n}}\}<\infty.$

When $w=\overline{0}^k$ for some $k\ge 1$, we have $h(X_w)=\ln(\lambda_w)$ where $\lambda_w$ is the real root of 
\[
t_w(z)=1-\sum_{n\in\CS,n<k}\frac{1}{z^{n+1}}=1-T_{\CS}(z)+\sum_{n\in\CS,n\ge k}\frac{1}{z^{n+1}}.
\]
Let $h_k(z)=1-T_{\CS}(z)+\frac{1}{z^k(z-1)}$. Then $t_w(z)\le h_k(z)$. We can get the exponential decay as $k\to\infty$ using similar arguments as before.
\end{proof}

\subsection{Perturbations of $d-$gap shifts}\label{subsec:d-gap}
Let $d\ge 1$. The $d$-gap shift is an $\CS$-gap shift where $\CS=\{kd\mid k\in\N_0\}$. We denote the shift by $X=X_d$. Clearly, $X_d$ is a sofic shift. In fact, it is easy to check that it has a right-resolving presentation $G_A$ where $A$ is a matrix of order $d$ given by,
 \[
 A=\begin{pmatrix}
        1&1&0&\dots&0\\
        0&0&1&\dots&0\\
        \vdots&\vdots&\vdots&\ddots&\vdots\\
        0&0&0&\dots&1\\
        1&0&0&\dots&0       
    \end{pmatrix}.
 \] 
The labeled graph representation of a $4$-gap shift is given in~\cref{fig:4}. 

\begin{figure}[h]
		\centering
		$\displaystyle
            \begin {tikzpicture}[-latex ,auto ,node distance =2cm and 3cm ,semithick , state/.style ={draw, circle}] 
		\node[state,scale=0.9] (A) {$v_0$};
		\node[state,scale=0.9] (B) [right of= A] {$v_1$};
        \node[state,scale=0.9] (C) [right of= B] {$v_2$};
        \node[state,scale=0.9] (D) [right of= C] {$v_3$};
        
		\path (A) edge [ left =25] node[above] {$0$} (B);
        \path (B) edge [ left =25] node[above] {$0$} (C);
        \path (C) edge [ left =25] node[above] {$0$} (D);
	\path (D) edge [bend right = -30] node[below] {$0$}(A);
		\path (A) edge[in=145,out=-145,loop,distance=2cm]node[left]{$1$} (A);
       
	\end{tikzpicture} 
	$
	\caption{Labeled graph representation for 4-gap shift}
	\label{fig:4}
\end{figure}
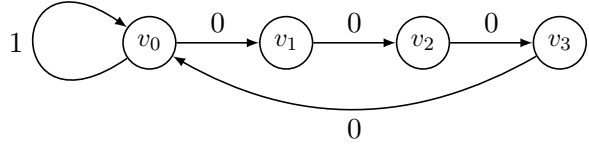

Hence, $X$ is a factor of the edge shift $X_A$. Let $\ell:\CL(X_A)\to\CL(X)$ be the block code. As done in~\cref{sec:sofic}, We use this presentation to effectively approximate the entropy of the perturbations of $d$-gap shift. This includes, in particular, perturbations by words starting and ending with 0, a previously unresolved case for general $\CS$-gap shifts.

Let $w\in\CL(X),$ and $\tilde{w}$ be defined as before. Clearly, $\tilde{w}$ has a unique pre-image, except when $\tilde{w}=\overline{0}^{kd}$ for some $k\ge 1$. Also, for $\tilde{w}\ne\overline{0}^{kd}$ for some $k\ge 1$, the unique word $u\in\ell^{-1}(\tilde{w})$ satisfies $s(u)=r(u)=v_0$ (as in the figure).  
We have the following result for $h(X_w)$. 

\begin{corollary}
    Let $d\ge 1$ and $w\in\mathcal{L}(X_d)$. Then, $h(X_w)=\ln(\lambda_w)$ where $\lambda_w$ is the largest real zero of
    \[
    g_w(z)= (z^d-z^{d-1}-1)c_u(z)+z^{d-1},
    \] where $u\in \CL(X_G)$ is such that $u\in\ell^{-1}(\tilde{w})$.
\end{corollary}
\begin{proof}
When $w\ne\overline{0}^{kd}$ for some $k\ge 1$, without loss of generality, assume that $\tilde{w}=w$. Let $\CG=(G,\ell)$ be the labeled graph representation of $X$ as defined above by the matrix $A$. Let $v_0$ be the unique vertex of $\CG$ that has a loop labeled $1$. Then, there exists a unique $u\in\CL(X_G)$ such that $s(u)=r(u)=v_0$ and $u=\ell^{-1}(w)$. Since $X_G$ is a subshift of finite type, by ~\cref{eq:SFT_one}, $h(X_w)=\ln(\lambda_w)$ where $\lambda_w$ is the largest real zero of
\[
 \det(zI-A)c_u(z)+Adj_{v_0,v_0}(zI-A)=(z^d-z^{d-1}-1)c_u(z)+z^{d-1}.
\]

When $w=\overline{0}^{n}$ for some $n=kd$, $k\ge 1$, then $X_w=X_{\CS}$ where $\CS_0=\{0,d,\dots,kd\}.$ Here, $h(X_w)=\ln(\lambda_w)$ where $\lambda_w$ is the largest real root of 
\[
\sum_{i=0}^k\frac{1}{z^{di+1}}-1=\frac{z^{(k+1)d+1}-1}{z^{kd}(z^d-1)}-1.
\]
We can easily verify that it is the largest zero of $(z^d-z^{d-1}-1)(z^{kd-1}+z^{(k-1)d-1}+\dots+z^{d-1})+z^{d-1}$. On the other hand, for any $u\in\ell^{-1}(w)$, we have $c_u(z)=z^{kd-1}+z^{(k-1)d-1}+\dots+z^{d-1}$.

\end{proof}


\begin{remarks}
    \begin{enumerate}
        \item When $w\ne 0u0$ for some $u$, this expression can also be obtained using~\cref{cor:S_gap}. 
        Imitating the proof of~\cref{thm:Sgap_exp}, we get an exponential decay of the entropy.
        \item \emph{For $w\in\mathcal{L}(X_d)\setminus\{1,10,\dots,1\overline{0}^d,01,\dots,\overline{0}^d1 \}$, we have $0<h(X_w)<h(X)$:} Note that $c_w(1)=1$ if and only if $c_w(z)=z^n$ where $n+1$ is the length of $w$. In this case, for $n\ge d+1$, $g_w(1)=0$, $g_w'(1)<0$ and $\lim_{z\to\infty}g_w(z)=\infty$. Hence $\lambda_w>1.$
  When $c_w(z)\ne z^n$ for any $n\ge 0$,  
  we can check that $g_w(1)=1-c_w(1)< 0$ and $g_w(\lambda_d)=\lambda_d^{d-1}>0.$ 
  \item \emph{$\min_{w\in\CL_n(X)}(h(X_w))=h(X_u)$ where $u=\overline{1}^{n-d}\ \overline{0}^{d}$:} Note that for any $w\in\CL(X)$ and $z>1$, $g_u(z)\le g_w(z)$.  
  
  \end{enumerate}
\end{remarks}

\section{Multi-word perturbations of subshifts of finite type}\label{sec:multi-wordSFT}
In this section, we consider the perturbations of subshifts of finite type by more than one word. 
From Section~\ref{sec:sofic}, clearly, in order to study the drop in entropy of perturbations of the sofic shift by an allowed word, it is important to know the drop in entropy of perturbations of subshifts of finite type by multiple allowed words. In the following result, we give an explicit expression. 

Let $A$ be an $N\times N$ irreducible matrix of non-negative integers, $G_A$ be the associated directed graph, and $X=X_A$ be the irreducible subshift of finite type associated with $G_A=(V,E)$, where $V=\{1,\dots,N\}$. For a word $w\in\CL(X)$, let $s(w),r(w)\in V$ be the source and range maps of $w$, respectively.  

Let $\CK=\{w_1,\dots,w_k\}\subseteq\CL(X)$ be a reduced collection of allowed words, that is, for $i\ne j$, $w_i\not\prec w_j$. We denote, $X_{\CK}$ to be the perturbation of $X$ by $\CK$. Let $f(n)=\#\CL_n(X_{\CK})$ and $F(z)=\sum_{n=0}^\infty f(n)z^{-n}$. We have the following result on the rational form of $F(z)$. For $1\le s,t\le N$, let $\chi_{s,t}=1$ if $s=t$ and $\chi_{s,t}=0$, otherwise. 
\begin{theorem}\label{thm:SFT_Multi}
    $F(z)=\sum_{i=1}^NP^{-1}(z)_{i,1}$ where 
    \[
    P(z)= \begin{pmatrix}
        zI-A&R(z)\\S(z)&-z\CM(z)
    \end{pmatrix},
    \] where $I$ is the identity matrix of order $N$, $R(z)$ is an $N\times k$ matrix such that $R(z)_{i,j}=z\chi_{i,r(w_j)}$, $S(z)$ is a $k\times N$ matrix such that $S(z)_{i,j}=\chi_{j,s(w_i)}$ and $\CM(z)_{i,j}=(w_j,w_i)_z$. Moreover, $h(X_{\CK})=\ln(\lambda_{\CK})$ where $\lambda_{\CK}$ is the largest real pole of $F(z)$. 
\end{theorem}
\begin{proof}
For $n\ge 0$ and $1\le j\le N$, let $f(n,j)$ be the number of words in $\CL_n(X_\CK)$ that end at the vertex $j$, when these words are considered as walks in $G_A$. For $v\in\CK$, let $g(n,v)$ be the number of words in $\CL_n(X)$ that end with $v$ and do not contain any appearance of words from $\CK$ except $v$ at the end. We assume $f(0,1)=1, f(0,j)=0$ for $1<j\le N$, and $g(0,v)=0$ for all $v\in\CK$. For $1\le i\le N$, let $F_i(z)=\sum_{n=0}^\infty f(n,i)z^{-n}$ and for $v\in\CK$, let $G_v(z)=\sum_{n=0}^\infty g(n,v)z^{-n}$.

We get the recurrence relations using arguments similar to those in~\cite{Combinatorial}. 
For each word counted in $f(n,j)$, we attach all the symbols corresponding to the edges from $j$ to $i$ at the end. There are $f(n,j)A_{j,i}$ many such words. These words are either counted in $f(n+1,i)$ or counted in $g(n+1,v)$ for some $v\in\CK$ where $r(v)=i$. Hence for each $1\le i\le N$, 
\[
\sum_{j=1}^Nf(n,j)A_{j,i}=f(n+1,i)+\sum_{v\in \CK}\chi_{i,r(v)}g(n+1,v).
\]
 Now, for the second set of equations, fix $v\in\CK$. Consider a word counted in $f(n,j)$ and attach $v$ at the end. This attachment is admissible only if $\chi_{j,s(v)}=1$. By tracking the first place a word from $\CK$ appears, we get terms from the nonempty correlations. Hence,   \[
        \sum_{j=1}^Nf(n,j)\chi_{j,s(v)}=\sum_{t\in(u,v)}g(n+t,u).
   \] We get the following system of linear equations by multiplying these equations by $z^{-n}$ and adding from $0$ to $\infty$;
   \[
   P(z)\begin{pmatrix}
       F_1(z)\\\vdots\\ F_N(z)\\G_{w_1}(z)\\\vdots\\ G_{w_k}(z)
   \end{pmatrix}=\begin{pmatrix}
       z\\0\\\vdots\\0
   \end{pmatrix}.
   \] Using the formula for the radius of convergence of $F(z)$, it is very easy to see that if $\lambda_{\CK}$ is the largest real pole of $F(z)$, then $h(X_{\CK})=\ln(\lambda_{\CK})$. 
\end{proof}
\begin{remark}
    When $X$ is a full shift of $N$ symbols, we have $A=(N)$ and we get back the expression as in~\cite{Combinatorial} where $F(z)=\dfrac{1}{z-N+s(z)}$ where $s(z)$ is the sum of entries of $\CM^{-1}(z)$. When $X$ is an irreducible SFT and $\#\CK=1$, we get back~\cref{eq:SFT_one}.
\end{remark}

\begin{remark}\label{rem:SFT_Exp}

By~\cref{thm:SFT_Multi}, in order to get the entropy of the perturbed system, we are interested in finding the largest real root of determinant of $P(z)$. As $P(z)$ has a block form, we have
\[
\det(P(z))=m(z)^{-k+1}\det(-zm(z)\CM(z)+Q(z)),
\] where $m(z)=\det(zI-A)$ is the characteristic polynomial of $A$ and $Q(z)=S(z)\text{Adj}(zI-A)R(z)$. Since the largest real root $\lambda_{\CK}$ is strictly smaller that the Perron root $\lambda_A$ of $A$, we are interested in the largest real root of $\det(-zm(z)\CM(z)+Q(z))$. Expanding this, we get
\[\det(-zm(z)\CM(z)+Q(z))=\det(-zm(z)\CM(z))+L(z)=(-zm(z))^k\det(\CM(z))+L(z),\] 
where $L(z)$ are the rest of the terms. Write, $m(z)=(\lambda_A-z)m'(z)$ where $m'(\lambda)\ne 0$ (this is possible as $\lambda_A$ is a simple root of $A$). Then, at $z=\lambda_{\CK}$, 
\[
(\lambda_A-\lambda_{\CK})^k=\frac{-L(\lambda_{\CK})}{(-\lambda_{\CK})^k\det(\CM(\lambda_{\CK}))}.
\]

Note that the diagonal entries of $\CM(z)$ has leading term $z^{n-1}$ and off-diagonal entries have leading term at most $z^{n-2}$. Using this, 
    one may show that $|\lambda_A-\lambda_\CK|=O(\lambda_A^{-n})$ for $n$ large which will imply $h(X_\CK)-h(X)=O(\lambda_A^{-n})$ for $n$ large enough. The argument is similar to the proof of~\cref{thm:Sgap_exp}. However, due to the extreme complexity of the calculations involved in this approach, we do not present a detailed proof. The provided details merely serve to validate the exponential decay of the drop in entropy. Nevertheless, in the following section, we provide a special case of this when the ambient space is a full shift. 
\end{remark}
\subsection{Exponential decay rate for multi-word perturbation of full shift}
Let $X$ be the full shift on $N$ symbols ($\CA=\{1,\dots,N\}$).
We look at the special case of multi-word perturbation on $X$ and explore the exponential decay rate of drop in entropy when the length of the forbidden words increases. 

Let $(\CK_n)_{n\in\N}$ be a sequence where each $\CK_n=\{w_{1,n},\dots,w_{k,n}\}$ is a collection of $k$ many words in $\mathcal{L}_n(X)$ such that $w_{i,n}$ is a prefix of $w_{i,n+1}$ for all $n\ge 1$ and $i=1,\dots,k$. For each $n$, let $h(X_{\CK_n})=\ln(\lambda_n)$. Then $\lambda_n$ is the largest (real) root of $z-N+s_n(z)$ where $s_n(z)$ is the sum of entries of $\mathcal{M}_{n}(z)^{-1}=\mathcal{M}_{\CK_n}(z)^{-1}$. 

As $n$ goes to infinity, each $i$ gives a unique infinite sequence in $\CA^{\N}$ and we denote it by $x_i\in \CA^{\N}$.
Also, by~\cref{rem:SFT_Exp}, it is expected that $h(X_{\CK_n})$ converges to $h(X)$ exponentially, as $n\to\infty$.
Hence, we are interested in finding 
\[
\lambda(x_1,\dots,x_k):=\lim_{n\to\infty}\dfrac{h(X)-h(X_{\CK_n})}{N^{-n}}=\lim_{n\to\infty}\dfrac{\ln(N)-\ln(\lambda_n)}{N^{-n}},
\] if the limit exists.

In this section, we obtain an explicit expression for $\lambda(x_1,\dots,x_k)$. Let $S:\CA^{\N}\to\CA^{\N}$ be the usual shift map. A point $x\in\CA^{\N}$ is \emph{periodic}, if there exists $s>0$ such that $S^sx=x$. For $x,y\in\CA^{\N}$, whenever we mention that there exists $m>0$ such that $S^mx=y$, we always mean that $m$ is the least positive integer with this property. 
The following key lemma will be used to prove the main result. 
\begin{lemma} \label{lemma:lim}
	With the notations as above, for $z>1$ and $1\le i\ne j\le k$,
	 \[ \lim_{n\to\infty}z^{-(n-1)}(w_{i,n},w_{i,n})_z=\begin{cases} 
	1 & \text{if $x_i$ is non-periodic} \\
	\frac{1}{1-z^{-s}} & \text{if $x_i$ is periodic with period $s$},
	\end{cases}
	\]
	\[ \lim_{n\to\infty}z^{-(n-1)}(w_{i,n},w_{j,n})_z=
	\begin{cases} 
	0 & \text{if $S^m x_i\ne x_j$ for any $m>0$} \\
	z^{-m} & \text{if $S^m x_i=x_j$ for some $m>0$ and $x_i,x_j$ are non-periodic} \\
	\frac{z^{-m}}{1-z^{-s}} & \parbox{11cm}{if $S^mx_i=x_j$ for some $m>0$, $x_i$ is non-periodic and $x_j$ is periodic with period $s$}\\
	\frac{z^{-m}}{1-z^{-t}} & \text{if $S^mx_i=x_j$ for some $m>0$ and $x_i$ is periodic with period $t$} .
	\end{cases}\]
\end{lemma}
\begin{proof}
	Fix $1\le i\ne j\le k$. Let $x=x_i$ and $y=x_j$, $u_n=w_{i,n}$, and $w_n=w_{j,n}$.
	Consider the correlation polynomials $a_n(z)=(u_n,u_n)_z$ and $b_n(z)=(u_n,w_n)_z$. The limit of $z^{-n+1}a_n(z)$ is shown in~\cite[Theorem 4.6.1]{BY}. However, we provide a proof for completion. \\	
	Let the correlation polynomial of $u_n$ with itself be
	\[
	a_n(z)=z^{n-1}+\sum_{i=1}^{n-1}a_{n,i}z^{n-1-i},
	\]
	where $a_{n,i}$ is either 0 or 1. When $x$ is non-periodic, $\lim_{n \to \infty}a_{n,i}=0$ for all $i$, and this implies that $\lim_{n \to \infty}\sum_{i=1}^{n-1}a_{n,i}z^{-i}=0$. Similarly when $x$ is periodic with period $s$, $\lim_{n \to \infty} \sum_{i=1}^{n-1}a_{n,i}z^{-i}=\sum_{j=1}^\infty z^{-js}=\frac{1}{1-z^{-s}}-1$, hence, we have the required limit. 
    
	Let $b_n(z)=\frac{1}{z}\sum_{i=1}^{n-1}b_{n,i} z^{n-i}$ be the cross-correlation $(u_n,w_n)_z$. We consider different cases. 
	
\noindent	\textbf{Case 1}. \emph{$S^m x\ne y$ for any $m>0$}: As $S^mx\neq y$ for any $m>0$, there exists $N_m$ such that $x_{m+1}x_{m+2}\dots x_{m+N_m}\neq y_1y_2\dots y_{N_m}$. Hence, for every $n\ge N_m$, we have $b_{n,m}=0$. Therefore,
$\lim_{n\to\infty}z^{-n+1}b_n(z)=0.$

\noindent \textbf{Case 2}. \emph{$x,y$ are non-periodic and there exists $m>0$ such that $S^mx=y$}:  Assume that $m$ is the least positive integer where $S^mx=y$. In this case, $b_{n,m}=1$ for all $n>0$. Let $m'>m$ be such that $S^{m'}x=y$. Then $S^{m'-m}y=S^{m'-m}S^mx=S^{m'}x=y$ implies that $y$ is periodic and that is a contradiction. Therefore, for $m\ne m'$, there exists $N_{m'}>0$ such that for $n\ge N_{m'}$, we have $b_{n,m'}=0$. Therefore  
$\lim_{n\to\infty}z^{-n+1}b_n(z)=z^{-m}.$

\noindent \textbf{Case 3}. \emph{$x$ is non-periodic, $y$ is periodic with period $t$ and there exists $m>0$ such that $S^mx=y$}: Let $t$ be the least period of $y$. Hence, if $S^\ell y=y$ for some $\ell>0$, then $\ell=pt$ for some $p\ge 1$. Also, assume that $m$ is the least positive integer where $S^mx=y$. 
In this case, $b_{n,m}=1$ for all $n>0$.
Let $m'>m$ be such that $S^{m'}x=y$. Then $S^{m'-m}y=S^{m'-m}S^mx=S^{m'}x=y$. Hence $m'=m+pt$ for some $p\ge 1$. That means for any $m'$ that is not of the form $m+pt$ for some $p\ge 1$, there exists $N_{m'}>0$ such that for $n\ge N_{m'}$, we have $b_{n,m'}=0$. Therefore,  
\[\lim_{n\to\infty}z^{-n+1}b_n(z)=z^{-m}\sum_{j=0}^\infty z^{-jt}=\frac{z^{-m}}{1-z^{-t}}.\]

\noindent \textbf{Case 4}. \emph{$x$ is periodic with period $s$, and there exists $m>0$ such that $S^mx=y$}: Let $s$ be the least period of $x$ and $m$ be the least positive integer where $S^mx=y$. Observe the following in this case.  
 
 a) $m< s$; otherwise, $y=S^{m-s}S^s x=S^{m-s}x$ contradicts the fact that $m$ is the least integer with this property. 
 
 b) $S^sy=S^{s+m}x=S^mx=y$ implies $y$ is periodic. If $t$ is the period of $y$, then $t\le s$. 
 
 c) $S^{s-m}y=S^{s-m+m}x=x$. Let $\ell\le s-m$ be such that $S^\ell y=x$. Then $S^{\ell+m}x=x$ implies $s\le \ell+m$. Hence $\ell=s-m$. Hence, $s-m$ is the minimum integer such that $S^{s-m}y=x$. 
 
 d) $S^t x=S^tS^{s-m}y=S^{s-m}y=x$ implies $s\le t$ as $s$ is the least period of $x$. Therefore $s=t$. That is $y$ is a periodic point with period $s$. 
 
 Hence, similar to the argument in the last case, if $m'>m$ is such that $S^{m'}x=y$, then $m'=m+ps$ for some $p\ge 1$. Hence for any $m'$ that is not of the form $m+ps$ for some $p\ge 1$, there exists $N_{m'}>0$ such that for $n\ge N_{m'}$, we have $b_{n,m'}=0$. Therefore,  
 \[\lim_{n\to\infty}z^{-n+1}b_n(z)=z^{-m}\sum_{j=0}^\infty z^{-js}=\frac{z^{-m}}{1-z^{-s}}.\]
\end{proof}
 Let $s_n(z)$ be the rational function that is the sum of the entries of $\CM_n^{-1}(z),$ where $\CM_n(z)=[(w_{j,n},w_{i,n})_z]_{i,j}$ is the correlation matrix function. Let $\alpha(z)=\lim_{n\to\infty}z^{-n}\CM_n(z)$. Firstly, we have the following result.   	

\begin{lemma}\label{lemma:inv_alpha}
    $\alpha(z)$ is invertible for all $z>\#\CK$.
\end{lemma}
\begin{proof}
Let $\#\CK=k$. Then $\alpha(z)$ is a matrix of order $k$. 
    Note that $\alpha(z)=\frac{1}{z}\beta(z)$ where the diagonal entries of $\beta(z)$ are of the form $1+\frac{z^{-m}}{1-z^{-\ell}}$ and the off-diagonal entries are of the form $\frac{z^{-m}}{1-z^{-\ell}}$ where $m,\ell\in\{0,\dots,\infty\}$ (we take $z^{-\infty}=0$). Also, for $z>1$, $\frac{z^{-m}}{1-z^{-\ell}}=\frac{z^{\ell-m}}{z^\ell-1}\le \frac{z^{\ell-1}}{z^\ell-1}\le \frac{1}{z-1}.$
    Hence, $\alpha_{i,i}(z)\ge 1\ge \frac{k-1}{z-1} \ge \sum_{j\ne i}\alpha_{i,j}(z).$ That is, $\alpha(z)$ is a strictly diagonally dominant matrix. This implies that $\alpha(z)$ is invertible. 
\end{proof}

\begin{lemma} Let $\alpha(N)$ be invertible. 
	With the notations as above, 
	\begin{equation}\label{eq:local_erate}
	\lambda(x_1,x_2,\dots,x_k)=
	\lim_{n\to\infty}\left(N^{n-1}s_n(N)+N^{n-2}s^2_n(N)+\dots\right).
	\end{equation}
\end{lemma}
\begin{proof}
 For each $n$, let $h(X_{\CK_n})=\ln(\lambda_n)$. Then, $\lambda_n$ is the largest real zero of $z-N+s_n(z)$. That is, $N-\lambda_n=s_n(\lambda_n)$. By mean value theorem, 
 \[
 N-\lambda_n=\frac{s_n(N)}{1-s_n'(c_n)},
 \] for some $\lambda_n<c_n< N$. Using the expansion for the logarithm, we get that 
 \[
-\ln\left(\frac{\lambda_n}{N}\right)=\frac{s_n(N)}{N(1-s_n'(c_n))}+\left(\frac{s_n(N)}{N(1-s_n'(c_n))}\right)^2+\dots.
 \] We will show that $\lim_{n \to \infty}s_n'(c_n)=0$ and thus the result will follow.  
 
 Consider $z^{-n}\CM_n(z)$. Each entry in this matrix is of the form $z^{-n}(w_{i,n},w_{j,n})_z$. Let $\alpha_{i,j}(z)=\lim_{n \to \infty}z^{-n}(w_{i,n},w_{j,n})_z$. The form of $\alpha_{i,j}(z)$ is obtained in Lemma~\ref{lemma:lim} for $z>1$. Therefore, 
 \[
 \lim_{n \to \infty}\det(z^{-n}\CM_n(z))=\lim_{n \to \infty}z^{-kn}\CD_n(z)=\det(\alpha(z)),
 \] where $\CD_n(z)$ is the determinant of $\CM_n(z)$ and $\alpha(z)=[\alpha_{j,i}(z)]_{i,j}$. Choose the largest interval $[z_0,N]$ for some $z_0>1$ such that $\det(\alpha(z))\ne 0$ for $z\in[z_0,N]$.   
 
 Let $\mathcal{S}_n(z)$ denote the sum of entries of the adjoint matrix of $\CM_n(z)$. Assume that $\CD_n(z)=\sum_{j=0}^{k(n-1)}a_{j,n}z^j$ and $\mathcal{S}_n(z)=\sum_{j=0}^{(k-1)(n-1)}b_{j,n}z^j$. Denote $a_n=\max_j\{|a_{j,n}|\}$ and $b_n=\max_j\{|b_{j,n}|\}$. Here
 \begin{eqnarray*}
 |s_n'(z)|&\le& \frac{|\CD_n(z)S_n'(z)|+|\CD_n'(z)\mathcal{S}_n(z)|}{(\CD_n(z))^2}< kna_nb_n\frac{z^{(n-1)(2k-1)+1}}{(z-1)^2(\CD_n(z))^2},
 \end{eqnarray*}
 since $|\CD_n(z)|< a_n\frac{z^{k(n-1)+1}}{z-1}, |S_n(z)|<b_n\frac{z^{(k-1)(n-1)+1}}{z-1},|\CD_n'(z)|< a_nk(n-1)\frac{z^{k(n-1)}}{z-1}$, and $|S_n'(z)|<b_n(k-1)(n-1)\frac{z^{(k-1)(n-1)}}{z-1}$.

Expanding the determinant via the Leibniz formula, we have
\[
\CD_n(z) = \sum_{\sigma \in S_k} \text{sign}(\sigma) \prod_{i=1}^k \CM_{n_{i,\sigma(i)}}(z).
\]
Each product involves $k$ polynomials of degree at most $n-1$ with coefficients 0 or 1, so the coefficient of any polynomial in such a product is bounded by $O((n-1)^k)$, and there are $k!$ such permutations. Summing over all permutations gives $a_n=O(k! n^k)$. Using a similar argument for the determinant of the submatrices, we get $b_n=O(k!  k  n^{k-1})$.
Hence, there exists $M_0$ such that $a_nb_n\le M_0n^{2k-1}$. 

Multiplying the last term in the above inequality by $\frac{z^{2kn}}{z^{2kn}}$ and using $\lim_{n \to \infty}z^{-kn}\CD_n(z)=\det(\alpha(z))$, we get that 
\[
|s_n'(z)|\le M(z)\frac{n^{2k}}{z^n}\rightarrow 0
\] as $n$ tends to infinity for $z\in[z_0,N]$, where $M(z)=\frac{kM_0}{(\det(\alpha(z)))^2}\frac{z^{2-2k}}{(z-1)^2}$. By~\cref{thm:multi_SFTLinear}, $\lim_{n\to\infty}\lambda_n=N$ and thus $c_n\in[z_0,N]$ for $n$ large enough.
Therefore, $\lim_{n \to \infty}s_n'(c_n)=0$.
\end{proof}

\begin{theorem}\label{thm:Exp_decay}
Let $\alpha(N)$ be invertible. 
    Then, \[\lambda(x_1,x_2,\dots,x_k)=\frac{T}{N}\] where $T=T(x_1,x_2,\dots,x_k)$ is the sum of entries of $\alpha^{-1}(N)$.
\end{theorem}
\begin{proof}
As before, let $\alpha_{i,j}(z)=\lim_{n\to\infty}z^{-n}(w_{i,n},w_{j,n})_z$ for all $1\le i,j\le k$ and $\alpha(z)=[\alpha_{j,i}(z)]_{i,j}$. Then $\lim_{n\to\infty}z^{-n}\CM_n(z)=\alpha(z)$ implies $\lim_{n\to\infty}z^{n}\CM_n(z)^{-1}=\alpha^{-1}(z)$. Hence, if $s_n(z)$ denotes the sum of entries of $\CM_n^{-1}(z)$, then
\[
 \lim_{n\to\infty}N^{n-1}s_n(N)= \frac{T}{N}, 
\] where $T$ is the sum of entries of $\alpha^{-1}(N)$.
Since $\lim_{n \to \infty}s_n(N)=0$, we get that the rest of the terms in~\cref{eq:local_erate} go to zero, as $n$ tends to infinity. 
Hence, we obtain,
\[
\lambda(x_1,x_2,\dots,x_k)=\lim_{n\to\infty}N^{n-1}s_n(N)= \frac{T}{N}. 
\] 
\end{proof}
\begin{remark}
    By~\cref{lemma:inv_alpha}, $\alpha(N)$ is invertible whenever $k<N$. 
\end{remark}
Note that all the possible cases regarding the limit $\alpha_{i,j}(N)$ are already discussed in Lemma~\ref{lemma:lim}. Also, when the orbits of $x_1,x_2,\dots,x_k$ do not intersect with each other, then $\alpha(z)$ is a diagonal matrix and hence the limit is the average of the individual limit that we obtain by forbidding single words. 

\subsubsection{Application to ergodic theory} Let $\CA=\{1,\dots,N\}$ and $X=\CA^\N$ be a one-sided full shift. For $w\in\CL(X)$, let $C_w=\{x\in X\mid x \text{ starts with } w\}$ be the \emph{cylinder based at $w$}. The cylinders based at all possible words in $\CL(X)$ form a basis for the pro-discrete topology on $X$. 
Moreover, $X$ is equipped with uniform invariant probability measure $\mu$ where $\mu(C_w)=\frac{1}{N^k}$ where $|w|=k$. For a given collection $\CK$ of finite allowed words, each of length $k$, we define a \emph{hole} $\bigcup_{w\in\CK}C_w$. Note that the perturbation $X_{\CK}$ is the collection of all sequences in $X$ that do not enter the hole $\bigcup_{w\in\CK}C_w$ by the shift map. Define,  
\[
\Omega_n(\CK)=\{x\in X\mid S^i(x)\notin\bigcup_{w\in\CK}C_w, \ 0\le i\le n-1\}
\] and define $\CA^\N\setminus\Omega_n(\CK)$ to be the \emph{survival set}. The \emph{escape rate into $\bigcup_{w\in\CK}C_w$ by $S$} is defined as 
\[
\rho(\CK)=-\lim_{n\to\infty}\frac{1}{n}\mu(X\setminus\Omega_n(K)).
\] 
The escape rate is an important concept in open dynamics where a hole is introduced in the state space and we look at the average rate at which the orbits escape into this hole. In our setup, the quantity defined above is the escape rate where the hole is given by $\bigcup_{w\in\CK}C_w$. 
In~\cite{BY}, it was proved that $\rho(K)=h(X)-h(X_{\CK})$. 

It is also interesting to see how the escape rate decays when the size of the hole converges to zero~\cite{FP12}.
Consider a shrinking sequence of (possibly disconnected) holes $\left\{H_n=\bigcup_{w\in\CK_n}C_w\right\}_{n\ge 1},$ where $|w|=n$ for all $w\in\CK_n$, $\#\CK_n=k$ for $n\ge 1$ and $H_1\supseteq H_2\supseteq\dots$. As $n\to\infty$, these holes converge to a finite union of points, say $x_1,\dots,x_k\in X$. 

We define the \emph{local escape rate} into $x_1,\dots,x_k$ as 
\[
\rho(x_1,\dots,x_k)=\lim_{n\to\infty}\dfrac{\rho(\CK_n)}{\mu(\CK_n)}=\lim_{n\to\infty}\dfrac{h(X)-h(X_{\CK_n})}{kN^{-n}},
\] where $\mu(\CK_n)=\mu(\bigcup_{w\in\CK_n}C_w).$ We have the following result. 
\begin{corollary}
    With notations as before, whenever $k<N$, we have, $\rho(x_1,\dots,x_k)=\dfrac{T}{kN}.$
\end{corollary}


\bibliographystyle{abbrv} 
\bibliography{mybib}
\end{document}